\documentclass[11pt]{amsart}
   \usepackage{ amsmath,amssymb,verbatim,graphicx,amsthm}
\usepackage[top=1in, bottom=1in, left=1in, right=1in]{geometry}
\usepackage{enumerate}
\usepackage{cite}
 \usepackage{url}
 \allowdisplaybreaks[1]

\usepackage[sc]{mathpazo}
\usepackage[T1]{fontenc}

\newcommand{\vertiii}[1]{{\left\vert\kern-0.25ex\left\vert\kern-0.25ex\left\vert #1
    \right\vert\kern-0.25ex\right\vert\kern-0.25ex\right\vert}}

\theoremstyle{plain}
\newtheorem{theorem}{Theorem}[section]

\newtheorem{lemma}[theorem]{Lemma}
\newtheorem{cor}[theorem]{Corollary}
\newtheorem{mydef}[theorem]{Definition}

   \renewcommand{\bf}{\bfseries}

   \newcommand{\N}{\mathbb{N}}

   \renewcommand{\epsilon}{\varepsilon}

\theoremstyle{plain}

\usepackage{setspace}
 
\begin{document}
 \title{ Collatz map as a non-singular transformation}
   \author{I. Assani}
\thanks{Department of Mathematics, UNC Chapel Hill, NC 27599, assani@email.unc.edu.}
 \begin{abstract}
 Let $T$ be the map defined on $\N=\{1,2,3, ...\}$ by $T(n) = \frac{n}{2} $ if $n$ is even and by $T(n) = \frac{3n+1}{2}$ if $n$ is odd. Consider the dynamical system  $(\N, 2^{\N},  T,\mu)$  where $\mu$ is the counting measure. This dynamical system $(\N, 2^{\N}, T, \mu)$ has the following properties.
\begin{enumerate}
\item  There exists an invariant finite measure $\gamma$ such that 
$\gamma(A) \leq \mu(A) $ for all $A \subset \N.$
\item  For each function $f\in L^1(\mu)$ the averages 
$\frac{1}{N} \sum_{n=1}^N f(T^nx)$  converge for every $x\in \N$ to $ f^*(x)$ where $ f^* \in L^1(\mu).$
\end{enumerate}
We also show that the Collatz conjecture is equivalent to the existence of a finite measure $\nu$ on $(\N, 2^{\N})$ making the operator $Vf = f\circ T$ power bounded in $L^1(\nu)$ with conserrvative part $\{1,2\}.$ 
\end{abstract}
 \maketitle

  \section{Introduction}
The original Collatz conjecture states that if $S$ is the map defined on $\N$ by $Sn = \frac{n}{2}$ if $n$ is even and by $3n +1$ if $n$ is odd , then for each natural number $n$ there exists $k\in\N$ such that $T^kn = 1.$ This conjecture has been extensively studied . See the nice survey and analysis  on this subject done by J. Lagarias \cite{JLag1}, \cite{JLag2}. See also Y. Sinai \cite{YSin} and E. Akin \cite{EA}.  As noted in \cite{JLag1} S. Kakutani, S. Ulam and P. Erdos had been interested in this problem.  Several  attempts have been made after Lagarias survey \cite{JLag2} and some offering equivalent formulations. See \cite{BO}, \cite{KS}, \cite{GV}, \cite{BK}, \cite{KL}. An operator theoretic approach was presented in \cite{LP}.\\
   An equivalent and more convenient map to study this conjecture is defined as $Tn = \frac{n}{2} $ if $n$ is even and $T n = \frac{3n+1}{2}$ if $n$ is odd.  A cycle for $T$ is a sequence ${a, Ta, ....T^{i-1}a}$ where $a\in \N,$ $T^ja\neq a$ for $ 1\leq j \leq i-1$ and  $T^i a = a.$ Since $T1 = 2$ and $T2=1$ , $\{1,2\}$ is a cycle. \\
The Collatz conjecture is equivalent to the combination of the following conjectures. 
\begin{enumerate}
\item The only cycle for $T$ is $\{1, 2\}$ 
\item The orbit of every $n\in\N$ under $T$  is bounded.
\end{enumerate}

  The Collatz conjecture as stated is a recurrence property to the set $\{1,2\}$. More precisely it states that from any point $n \in \N\backslash \{1,2\}$ iterates $T^k(n)$ will after finite iterates go to the set 
$\{1, 2\}.$  Such recurrence properties have been studied quite extensively in Ergodic Theory in the context of measure preserving transformations but also in the context of  non-singular transformations which in our view fits better the Collatz map. We recall that  if $S$ is a measurable map from the $\sigma$-finite measure space $(X, \mathcal{A}, \nu)$ to itself then the set $(X, \mathcal{A}, \nu, S)$ is called a dynamical system. The map $S$ is said to be non-singular if  for all $A\in \mathcal{A}$ we have $\nu(S^{-1}(A)) =0$ if $\nu(A) =0.$
Such systems are called null preserving in \cite{Kreng} p.3. In this note we will focus on the measurable space $(\N, 2^{\N})$ where $2^{\N}$ denotes the power set of $\N,$ set of natural numbers without zero. The Collatz map $T$  is clearly measurable with respect to this measurable space. 
\begin{mydef}
 The set $(\N, 2^{\N}, T, \nu)$ with $\nu$ a $\sigma$-finite measure for which $T$ is non-singular is called the Collatz dynamical system with the measure $\nu$.
\end{mydef}
  A nonnegative measure $\nu$ on $(\N, 2^{\N})$ is defined by its values on $\N$ . It is entirely defined by the sequence of nonnegative numbers $\nu(n)$ where $n\in\N.$ We assume that $\nu(n) < \infty$ for each n. The measure $\nu$ is $\sigma$-finite is $\sum_{n=1}^{\infty} \nu(n) = \infty$ and finite if $\sum_{k=1}^{\infty} \nu(n) <\infty.$  A natural $\sigma$-finite measure on $(\N, 2^{\N})$
is given by the counting measure $\mu$ where $\mu(n) = 1$ for each $n\in \N.$ Since this counting measure has only the empty set as nullset the map $T$ is non-singular with respect this measure. Another measure $\theta$ on $(\N, 2^{\N})$ is equivalent to the counting measure if $\theta(n) >0$ for each $n\in\N.$ \\ 
Our main result in this short paper is the following.

\begin{theorem}
     Let  $(\N, 2^{\N}, T, \mu)$ be the Collatz dynamical system with the counting measure $\mu$. The following are equivalent.
\begin{enumerate}
 \item There exists a finite measure $\alpha$ equivalent to $\mu$ for which the dynamical system $(\N, 2^{\N}, T, \alpha)$ is power bounded in $L^1(\alpha)$ with conservative part $\{1,2\}.$
\item  For each $n\in\N$ there exists $k$ such that $T^k(n)\in \{1,2\}.$
\end{enumerate} 
\end{theorem}

Our main approach for the proof will come from Ergodic Theory.  More precisely we will consider the dynamical system $(\N, 2^{\N}, T, \mu)$ where $\mu$ is the counting measure, and study recurrence properties of the Collatz map $T$ on the measurable space $(\N, 2^{\N}).$ 
The paper is organized as follows.
\begin{enumerate}[(i)]
\item In the second section we recall some tools from ergodic theory that we will use later. These tools include Hopf's decomposition, power bounded non-singular transformations and aymptotically mean bounded non-singular transformations.  
\item In the third section we apply these tools to the dynamical system $(\N, 2^{\N}, T, \mu)$ and obtain a partition of $\N$ into three sets $C,$ $D_1$ and $D_2$. The set $C$ is an at most countable union of cycles. Elements in $D_1$  after finitely iterates of $T$ enter $C.$ Finally the invariant set $D_2$ under $T$ (ie. $T^{-1}(D_2)= D_2$) is the set of elements with unbounded trajectories. 
We establish ergodic properties of the dynamical system $(\N, 2^{\N}, T, \mu).$
This decomposition is actually valid for all map $V: \N \rightarrow \N$   
\item In the fourth section we prove that Collatz conjecture is equivalent to the search of a finite measure $\nu$ on $(\N, 2^{\N})$ making $T$ power bounded and non-singular.  
\end{enumerate} 
   \noindent{\bf Acknowledgments} We thank the anonymous referees for their helpful suggestions, comments and questions that have improved this paper. 
\section{Hopf Decomposition, recurrence property of power bounded non-singular transformations and asymptotically mean bounded non-singular transformations}
   \subsection{Hopf Decomposition} 
   In this subsection we recall some tools and definitions used in ergodic theory that we will apply in the next section to the dynamical system $(\N, 2^{\N}, T, \mu).$ 
\begin{mydef}
   Consider $(X, \mathcal{A}, \rho)$ a $\sigma$-finite measure space and $V$ a measurable map from $X$ to $X.$ The map $V$ is said to be non-singular with respect to $\rho$ if for each measurable subset of $\mathcal{A},$ we have   $\rho(V^{-1}(A)) = 0$ whenever $\rho(A) = 0.$ 
\end{mydef}
\begin{mydef}
   A set $C$ is said to be $V$-absorbing if $ C \subset V^{-1}(C).$   
\end{mydef}
  It is clear that $T$ is non-singular with respect to the counting measure $\mu,$ and to any other measure $\nu$ which is equivalent to $\mu.$ Actually for the measure space $(\N, 2^{\N}, \mu)$ the only null set is the empty set. 
\begin{mydef} 
  A measurable subset $W\in \mathcal{A}$ is said to be wandering if $V^{-i}(W)\cap V^{-j}(W) $ is empty whenever $i, j\in \N$ and $i\neq j.$  
\end{mydef} 
  Almost every point of a wandering set $W$ never returns to $W$ under $V$. This means that for $x\in W$ ,  $V^nx \notin W$ for each $n\in\N$. \\
The next theorem is the "Hopf decomposition" applied to a dynamical system $(X, \mathcal{A}, \rho, V)$ where $V$ is non-singular with respect to the $\sigma$-finite measure $\rho.$ A proof of the "Hopf's Decomposition Theorem" can be found in Krengel 's book \cite{Kreng} page 17.
\begin{theorem}\label{HD}(Hopf decomposition)
 Let $(X, \mathcal{A}, V, \rho)$ be a non-singular dynamical system. The space $X$ can be decomposed into two disjoint measurable subsets $C$ and $D$ with the following properties; 
\begin{enumerate}
\item $C$ is $V$-absorbing.
\item The restriction of $V$ to  $C,$ is conservative (the map $V$ returns (a.e.) to  every subset of $C$ with positive measure  infinitely often) . In particular if $\rho(m)>0$ for every $m\in C$ then there exists $k(m)\in \N$ such that  $V^{k(m)}(m) = m.$) 
\item The set $D= C^c$ is called the dissipative part. It is at most a countable union of wandering sets $W.$ 
\end{enumerate}
\end{theorem} 
 One can remark that "Hopf Decomposition" is the same for each measure equivalent to $\rho.$
  \subsection{ $L^1$ Power bounded non-singular transformation} 
The second tool from Ergodic Theory concerns power bounded non-singular transformations.
\begin{mydef}
   Let $(X, \mathcal{A}, V, \rho)$ be a non-singular dynamical system. It is power bounded in $L^1(\rho)$ if there exists a finite constant $M$ such that 
$\rho(V^{-n}(A)) \leq M \rho(A)\, \text{for all}\, A \in \mathcal{A}\, \text{and for all} \, n \in\N.$
\end{mydef}
 Power bounded non-singular dynamical systems on finite measure spaces have nice recurrence properties expressed in the next theorem.
\begin{theorem}\label{PB}
Let $(X, \mathcal{A}, \rho, V)$ be a non-singular dynamical system.  Assume that $\rho(X) <\infty$ and
that this dynamical system is power bounded in $L^1(\rho).$
Then 
\begin{enumerate}
\item There exists $v_0^* \in L_+^1(\rho),$  such that $\int \mathbb{1}_A v_0^*d\rho = \int \mathbb{1}_{A}\circ V v_0^*d\rho$ 
\item The conservative part $C$ of $(X, \mathcal{A}, \rho, V)$ is equal to the set where $v_0^*>0.$
\item For $\rho$ a.e. $x \in D$ there exists $m(x) \in \N$ such that $V^{m(x)}(x) \in C.$ 
\end{enumerate} 
 \end{theorem}
 The proof of this theorem can be found in \cite{DS} (Theorem 12, p.683) and \cite{IAs} (Theorem III.1). \\
 The function $v_0^*$ is obtained by considering the adjoint $U^*$ of the operator $U$ 
defined by $Uf = f\circ V.$ As a consequence of the mean ergodic theorem (see \cite{KY}) the averages $M_N(f) = \frac{1}{N}\sum_{n=1}^N f\circ V^n$ converge for each function $f\in L^1(\rho).$ We also have for each $g\in L^{\infty}(\rho)$ the mean convergence of the averages $M_N^*(g) = \frac{1}{N}\sum_{n=1}^N (U^*)^{n}(g)$ where $U^*$ is the adjoint operator of $U.$ As a consequence we have the convergence of the averages $M_N^{*}(\mathbb{1}_X)$ to the $U^*$ invariant function $v_0^*,$ and for each function $f\in L^1(\rho)$ we obtain the following equalities
\begin{equation}\label{MET}
\lim_N \int M_N(f) d\rho = \lim_N\int f .M_N^*( \mathbf{1}_X) d\rho = \int f.v_0^*d\rho.
\end{equation}
One can remark that the map $\gamma: A\in\mathcal{A}\rightarrow \gamma(A) = \int \mathbb{1}_A v_0^*d\rho$ is invariant with respect to $V$ (i.e. $\gamma(A) = \gamma (V^{-1}(A)$) and that 
$\gamma(X) = \rho(X)>0.$\\
 For a $L^1$ power bounded dynamical system defined on a finite measure almost every point $x\in D,$ (the dissipative part given by Hopf 's decomposition) eventually enters its  conservative part $C$. Indeed assume that this was not the case. Then we could find a set $B\in\mathcal{A}$ such that $\rho(B) >0 $ and $V^m(B) \subset D$ for each $m\in\N.$ Then we would have 
 $B\subset V^{-m}(D)$ for each $m\in \N$  and using (\ref{MET}) we would obtain the following contradiction 
$$0< \rho(B) \leq \frac{1}{N} \sum_{n=1}^N \rho( V^{-m}(D)) \rightarrow_{N} \int \mathbf{1}_{D}v_0^{*}d\rho = 0.$$  
 \subsection{Asymptotically mean bounded non-singular dynamical system}, \\
   In this section we gather one more tool from ergodic theory when the measure $\rho$ is $\sigma$-finite and $\rho(X)= \infty.$
\begin{mydef}
  Let $(X, \mathcal{A}, \rho, V)$ be a non-singular dynamical system.  We say that this system is asymptotically mean bounded in $L^1(\rho)$ if there exists a finite constant $M$ such that 
$$\limsup_N \frac{1}{N}\sum_{n=1}^N \int_{Y}\mathbb{1}_A\circ V^n d\rho \leq M \rho(A) $$ for each measurable set $A\in \mathcal{A}$ and for each $Y$ such that $\rho(Y) <\infty.$
\end{mydef} 
  It was shown in \cite{RyllN} that the system $(X, \mathcal{A}, \rho, V)$ is asymptotically mean bounded in $L^1(\rho)$ iff for each function $f\in L^1(\rho)$ the averages
  $$\frac{1}{N}\sum_{n=1}^N f(V^nx) \, \text{converge a.e. to }\, f^*(x), \,\, f^*\,  \in L^1(\rho)$$.
   We present one more result in this subsection.
 \begin{theorem}\label{AMB0}
Let $(X, \mathcal{A}, \rho, V)$ be an asymptotically mean bounded non-singular dynamical system with $\rho$ being a $\sigma$-finite measure.
\begin{enumerate}
\item  There exists an invariant $\sigma$ finite measure $\Delta$ such that 
$\Delta(A) \leq M \rho(A) $ for all $A\in \mathcal{A}.$
\item  For each function $f\in L^1(\rho)$ the averages 
$\frac{1}{N} \sum_{n=1}^N f(V^nx)$  converge a.e to $ f^*(x)$ where $ f^* \in L^1(\rho).$
\item If the system $(X, \mathcal{A}, \rho, V)$ satifies the conditions (1), (2) and (3) of Theorem \ref{PB} then the averages $\frac{1}{N}\sum_{n=1}^N f(V^nx)$ converge a.e. for every $f\in L^{\infty}(\rho).$
\end{enumerate}

\end{theorem}  
\begin{proof}
 The first two statements are consequences of results in \cite{RyllN}. The third statement can be derived from Theorem 1 in \cite{GH}.

\end{proof}



\section{Ergodic properties of the Collatz map}
   In this section we apply the ergodic tools gathered in the previous section to the dynamical system 
 $(\N, 2^{\N}, T, \mu)$, $T$ being the Collatz map and $\mu$ the counting measure ($\mu(i) = 1$ for all $i\in \N,$ $i\geq 1.$)
   We start with the Hopf decomposition.
  \subsection{Hopf's decomposition for the Collatz map}
   \begin{theorem}\label{HDC}
  Consider the dynamical system $(\N, 2^{\N}, T, \mu)$ .  There exists a partition of $\N$ into three sets $C$, $D_1$ and $D_2.$
\begin{enumerate}
\item
  The set $C$ is the conservative part given by Hopf's decomposition (Theorem \ref{HD}) . It is composed of at most a countable number of cycles $C_i, 1\leq i<\infty$ . The complement of $C,$ the dissipative part is partitioned into two subsets $D_1$ and $D_2.$
\item The set $D_1$ is equal to $\cup_{k=1}^{\infty} T^{-k}(C) \backslash C$. It is the set of elements of $D$ which enter $C$ after finitely many iterates of $T,$ into one of the cycles $C_i.$
\item We have $ T^{-1}(C\cup D_1) = C\cup D_1.$
\item The set $D_2$ is the complement of $C\cup D_1$ into $\N.$ It is the set of elements of $\N$ having an unbounded trajectory.
 The set $D_2$ is invariant, i.e. $T^{-1}(D_2) = D_2.$
\end{enumerate}
 \end{theorem} 
     \begin{proof}
      The fact that $C$ the conservative part is a countable union of cycles follows from the fact that on $C$ the map $T$ is recurrent meaning that it returns infinitely to any set of positive measure in $C$. Since each point has positive measure, this recurrence property creates cycles in $C$. Their number is clearly at most countable since $\N$ itself is countable.\\
       The set $C$ is $T$-absorbing in the sense that $ C\subset T^{-1}(C).$ The set $D_1$ is equal to $\cup_{j} T^{-j}(T^{-1}(C)\backslash C)$ which can be written $[\cup_{k=1}^{\infty}T^{-k}(C)] \backslash C$.  Therefore, $D_1$ is  the subset of $D = C^c$ composed of points not in $C$ which enter $C$ after finitely many iterates of $T.$ \\
      Since  $\cup_{j=0}^{\infty}T^{-j}(C) = C\cup D_1$ and the sequence $T^{-j}(C)$ is monotone increasing, we have 
$$T^{-1}(C\cup D_1) = T^{-1}[\cup_{j=0}^{\infty}T^{-j}(C)] = \cup_{j=1}^{\infty}T^{-j}(C) = \cup_{j=0}^{\infty}T^{-j}(C) = C\cup D_1.$$ 
As a consequence we can derive the equation $T^{-1}(D_2) = D_2$ if we denote by $D_2$ the complement of $C\cup D_1.$ 
      \end{proof}
\noindent{\bf Remark}
   The set $D1$ is not empty because $T^{-1}(\{1,2\}) = \{1,2,4\}$ and $4$ is in the dissipative part of $T.$
  \subsection{Pointwise convergence - Existence of an invariant measure}
     In this section we check that the tools listed in the previous section apply to the dynamical system $(\N, 2^{\N}, T, \mu)$. 
\begin{theorem} 
    The system $(\N, 2^{\N}, T, \mu)$ is asymptotically mean bounded in $L^1(\mu)$. 
\end{theorem}
\begin{proof}
  We need to show that 
\begin{equation}\label{AMB1}
\limsup_N \frac{1}{N}\sum_{n=1}^N \int_{Y}\mathbb{1}_A\circ T^n d\mu \leq M \mu(A). 
\end{equation}  
for each measurable set $A\in \mathcal{A}$ and for each $Y$ such that $\mu(Y) <\infty.$
 The condition $\mu(Y)<\infty$ implies that the set $Y$ is finite. Therefore, it is enough to show that (\ref{AMB1}) holds for each point $y \in \N.$ 
 We distinguish two cases . If  $\mu(A) = \infty$ then (\ref{AMB1}) is clearly true.  We can assume that $\mu(A) <\infty,$ or in other words $A$ is a finite subset of $\N.$  This observation allows to reduce the proof of (\ref{AMB1}) to the case where $Y = \{y\}$ and 
$A = \{a\}$. 
\begin{enumerate}
\item If $(y, a)\in D\times D$ then $\lim_N\frac{1}{N} \sum_{n=1}^N \mathbb{1}_{a}(T^ny)  = 0$ since $\{y\}$ and $\{a\}$ are wandering sets. 
\item If $(y,a)\in D_2\times C $ then $\lim_N\frac{1}{N} \sum_{n=1}^N \mathbb{1}_{a}(T^ny)  = 0$ because  the orbit $\{T^ny: , n\in\N\}$ is contained in the invariant set $D_2.$
\item If $(y,a) \in D_1\times C_j$ where $C_j$ is one of the cycles in $C,$ then \\
 $\lim_N\frac{1}{N} \sum_{n=1}^N \mathbb{1}_{a}(T^ny) = \frac{1}{\# C_j} \leq \mu(\{a\})$
\item If $(y, a)\in C_l\times C_j$ where $C_l$ and $C_j$ are cycles in $C,$ then \\
$\lim_N\frac{1}{N} \sum_{n=1}^N \mathbb{1}_{a}(T^ny) = \frac{1}{\# C_j} \leq \mu(\{a\})$ if $l=j$ and 
          $\lim_N\frac{1}{N} \sum_{n=1}^N \mathbb{1}_{a}(T^ny) = 0 $ if $l\neq j.$
\end{enumerate}
  In summary we have shown that 
 $$\lim_N \frac{1}{N}\sum_{n=1}^N \int_{\{y\}}\mathbb{1}_A\circ T^n d\mu =\lim_N \frac{1}{N} \sum_{n=1}^N \mathbb{1}_{a}(T^ny) \leq \mu(\{a\})$$
 By linearity we can conclude that 
 $$\lim_N \frac{1}{N}\sum_{n=1}^N \int_{Y}\mathbb{1}_A\circ T^n d\mu \leq \mu(A), $$ for each measurable set $A\in \mathcal{A}$ and for each $Y$ such that $\mu(Y) <\infty.$ In other words the dynamical system $(\N, 2^{\N}, T, \mu)$ is asymptotically mean bounded in $L^1(\mu).$ The constant $M$ is equal to 1.

\end{proof} 
 As a consequence of Theorem \ref{AMB0} we obtain for the dynamical system $(\N, 2^{\N}, T, \mu)$ the properties listed in the previous section. The next theorem is part of the abstract.

\begin{theorem}\label{Th3.3}
  The dynamical system $(\N, 2^{\N}, T, \mu)$ has the following properties.
\begin{enumerate}
\item  There exists an invariant finite measure $\gamma$ such that 
$\gamma(A) \leq \mu(A) $ for all $A \subset \N.$
\item  For each function $f\in L^1(\mu)$ the averages 
$\frac{1}{N} \sum_{n=1}^N f(T^nx)$  converge for every $x\in \N$  to $ f^*(x)$ where $ f^* \in L^1(\mu).$
\end{enumerate}
\end{theorem}
   \begin{proof}
     Compared to Theorem \ref{AMB0}, we can eliminate the almost everywhere term in the second part of the theorem because the only nullset for $(\N, 2^{\N}, \mu)$ is the empty set. 
\begin{enumerate}
\item  One can make the value of $f^*$ more explicit by using  the computations made for the various cases for $(y,a)$. We have 
   for $f_A= \mathbb{1}_A$ with $\mu(A) <\infty,$  
\begin{equation}
f_A^{*} = \sum_{i=1}^{\infty} \mathbb{1}_{\left(\cup_{j=0}^{\infty} T^{-j}(C_i)\right)}\frac{\# (A\cap C_i)}{\#C_i}.
\end{equation} 
\item As a consequence of this equation one can find an equivalent finite invariant measure to $\Delta.$ 
       It suffices to take a measure $\beta$ finite equivalent to $\mu$ and integrate $f^{*}_A$ with respect to $\beta.$ We obtain 
      $\gamma (A) =\int f^*_A d\beta = \sum_{i=1}^{\infty} \beta\left(\cup_{j=0}^{\infty} T^{-j}(C_i)\right)\frac{\# (A\cap C_i)}{\#C_i}.$ We have 
    $\gamma(A) = \gamma(T^{-1}(A))$ for all $A\in 2^{\N}$ with $\mu(A) <\infty$, since $f^*_A\circ T = f^*_A.$  The measure being finite this last equality extends by continuity to $2^{\N}.$
    One can observe that $\gamma(A) = 0 $ if $A \subset D$ since in this case $\frac{\# (A\cap C_i)}{\#C_i}= 0,$ the sets $C_i$ being in the conservative part $C.$  In other words this measure $\gamma$ is supported on $C.$
 The measure $\gamma$ is finite because the sets $\left(\cup_{j=0}^{\infty} T^{-j}(C_i)\right)$ are disjoint and thus $$\sum_{i=1}^{\infty} \beta\left(\cup_{j=0}^{\infty} T^{-j}(C_i)\right) \leq \beta (\N) <\infty.$$
\end{enumerate} 
      One can choose $\beta$ in such way that  $\beta\left(\cup_{j=0}^{\infty} T^{-j}(C_i)\right) = 2^{-i}$ then we have 
$$\gamma (A) = \sum_{i=1}^{\infty} 2^{-i}\frac{ \nu_i(A)}{p_i}$$ where $\nu_i$ is the finite invariant measure with support the cycle $C_i$  and defined by \\
    $\nu_i(A)  = \#(A\cap C_i)$ and $p_i = \#(C_i)$ is the period of the cycle $C_i.$
\end{proof}

\noindent{\bf Remarks.}\\
   We can make the following remarks. The third remark leads to the next lemma and theorem \ref{Ivm}
\begin{enumerate}
\item  If the Collatz conjecture is true then the invariant measure $\gamma$ is uniform on the cycle $\{1,2\}$.
\item  If the conjecture is false in the sense that there are additional cycles then there are many invariant finite measures each barycentric averages of the uniform invariant  measures defined on these cycles
\item if the conjecture is false and  there is an unbounded orbit then we can construct on $D_2$ many $\sigma$-finite infinite invariant measures. Such construction is made in the following theorem. 
\end{enumerate} 
  We will need for it the following lemma. 
\begin{lemma}\label{L}
   Let us denote by $\N_2$ the set $\{k\in\N: 2k =1 \, \text{mod 3}\}$  and let  $a\in D_2$ . There exists an inifinite number of $k\in \N$ such $T^ka \in\N_2.$
\end{lemma}
\begin{proof}
The set $\N_2$ is also the set of natural numbers of the form $3p+2.$  We denote by $\N_1 = \{ 3p+1; p\in \N\}$ and by $\N_0=\{3p; p\in\N\}.$
  It is enough to show that if  we take an element in the orbit of $a$ let us say $q=T^ma\in \N_0\cup N_1$ there exists a natural number $s$ such that $T^{s+m}a \in\N_2.$
  We distinguish  two cases
\begin{enumerate}
\item  If $q= 3^k2^hn$ where $n\notin \N_0\cup 2\N$ then $T^hq = 3^kn$ which is odd . Thus $T^{h+1}q =\frac{1}{2} (3^{k+1}n+1) $ which belongs to
$\N_2.$
\item  If $q = 6p + r$ where $r=1, 4$ (the case $r=3$ implies $q \in \N_0$ treated in (1)) then $T(6p+1) = 9p+2 \in\N_2$ and $T(6p+4)= 3p+2\in\N_2.$
\end{enumerate}
These estimates prove the lemma.
\end{proof}
\begin{theorem}\label{Ivm}
If the Collatz conjecture is false then there exist in $D_2,$  $\sigma$-finite infinite invariant measures with support in $D_2.$ Furthermore these measures do not satisfy the condition (1)
of Theorem \ref{AMB0}.
\end{theorem}
\begin{proof}
First we denote by $\N_2= \{k\in\N; 2k =1 \, \text{mod 3}\}.$ These are the only natural numbers $k$ such that $\#\{T^{-1}(k)\}=2.$   
Take $a \in D_2$ then consider first the invariant set $\mathcal{F}=\cup_{j=0}^{\infty}T^{-j}\cup_{k=0}^{\infty}\{T^ka\}$. To construct an invariant measure $\theta$ it is enough to verify that at each element $x\in\mathcal{F}$ we have $\theta^{-1}(x) = \theta(x)$. We start by setting $\theta(a) =1.$ In the subtree generated by $\{a\}$ namely the set $\mathcal{T}_1= \cup_{j=0}^{\infty} T^{-j}( a)$ we distinguish two cases 
\begin{enumerate}
\item If $a\notin\N_2$ then $\#\{T^{-1}(a)\} =1$ and we define 
$\theta ( T^{-1}(a)) = 1$
\item If $a\in \N_2$ then $\{T^{-1}(a)\} =\{b_1, b_2\}$,  define 
$\theta(b_1) = \theta(b_2) = 1/2.$
\end{enumerate} 
We have $\theta(T^{-1}(a)) = \theta(a).$
 We proceed in a similar way along the subtree $\cup_{j=0}^{\infty} T^{-j}( a).$ To preserve the invariance property of $\theta$ along this subtree we consider $T^{-1}(b_1)$ and $T^{-1}(b_2).$ Again distinguishing the cases where the cardinality of these sets is one or two. For instance if $T^{-1}(b_1) = \{c_1, c_2\}$ then we set 
$\theta(c_1) = \theta(c_2) = \frac{1}{2}\theta(b_1).$ If $T^{-1}(b_2) =\{d\}$ then $\theta(d) = \theta(b_2)$. \\
Proceeding inductively along the subtree we define $\theta$ for each node of this subtree in such way that $\theta(T^{-1}(n) ) = \theta (n)$ for each $n$ in this subtree.\\
 Now we can define $\theta $ on the subtree $\mathcal{T}_2 = \cup_{j=0}^{\infty} T^{-j}(Ta).$
Here again we can distinguish two cases. 
\begin{enumerate}
\item If  $Ta \notin \N_2$  then we set $\theta(Ta) = \theta(a)=1.$
\item If $Ta\in \N_2$ then $T^{-1}(Ta) =\{a, e\}$ . We set $\theta(e) =1$ and
 $\theta (Ta) = 2.$ 
\item For the subtree $\cup_{j=0}^{\infty}T^{-j}(e)$ we can proceed as we did for the subtree
$\cup_{j=0}^{\infty}T^{-j}(a)$ to define $\theta$ on this subtree.
One can observe that the two subtrees $\mathcal{T}_1$ and $\mathcal{T}_2$ are disjoint since any node in the intersection would have two distinct images under $T.$\\ The set $\{T^ka: k\in \N\}$ containing no cycle because $D_2$ does not have any, we can proceed by induction on $T^j(a)$ and define an invariant measure $\sigma$-finite (infinite, since $$\sum_{k=0}^{\infty} \theta(T^ka)\geq \sum_{k=0}^n\theta(T^ka)\geq n+1$$)  on the invariant set  $\mathcal{F}.$
\end{enumerate}   
  To complete the proof we can use Lemma \ref{L}. In the orbit of $\{a\}$ under $T$ there are infinitely nodes in $\N_2.$  Therefore , $\limsup_k\theta(T^ka)= \infty$ and this violates the condition (1) in Theorem \ref{AMB0} since $M\mu(T^ka) = M.$  
\end{proof}

\section{Characterization of the Collatz conjecture through power bounded non-singular transformations}
   In this section we use some of the ideas in the proof of  Theorem \ref{HDC} to derive a characterization of the Collatz map through power bounded non-singular transformations. More precisely we would like to prove the following theorem by exploiting the fact that the map $T$ is onto and that $T^{-1}(k)$ is a singleton unless $2k = 1$ mod 3. 
\begin{theorem}\label{CJ}
     Let  $(\N, 2^{\N}, T, \mu)$ be the Collatz dynamical system with the counting measure $\mu$. The following are equivalent.
\begin{enumerate}
 \item There exists a finite measure $\alpha$ equivalent to $\mu$ for which the dynamical system $(\N, 2^{\N}, T, \alpha)$ is power bounded in $L^1(\alpha)$ with conservative part $\{1,2\}.$
\item  For each $n\in\N$ there exists $k$ such that $T^k(n)\in \{1,2\}.$
\end{enumerate} 
\end{theorem} 
\begin{proof}
   The first statement implies the second since as indicated in Theorem \ref{PB} for power bounded non-singular transformation points in $D$ enter after finitely many iterates the conservative part $C.$  The assumption in the second statement implies that $\{1,2\}$ is the conservative part of $(\N, 2^{\N}, T, \mu)$ (and for any measure equivalent to $\mu$). Simple considerations show that the set $T^{-1}(C) \backslash C$ is equal $\{4\}.$ Therefore $D$ is just the "tree" created by the inverse map $T^{-1}$ starting at $4$. This is the set $\cup_{j=0}^{\infty} T^{-j}(\{4\}).$ One can construct a power bounded transformation by starting with $\delta= \alpha(4)>0 $ then looking at the predecessors of $4$. Since $8 = T^{-1}(4)$ is the only predecessor of $4$ we set $\alpha(8) = \frac{1}{4}\delta.$ Then $8$ has two predecessors $16$ and $5$. We set $\alpha(16) = \frac{1}{4}\alpha(8)$ and $\alpha(5) = \frac{1}{4}\alpha(8)$. By induction we define $\alpha(n)$ for each $n\in D$ by being $\frac{1}{4}$ of the value of its predecessor in the tree.
 We can observe that with this process for any subset $A$ of the tree $\cup_{j=0}^{\infty} T^{-j}(\{4\})$ we have $\alpha (T^{-1}(A)) \leq \frac{1}{2}\alpha(A).$ Furthermore we can derive that for every $n\in\N$ $\alpha(T^{-n}(A)) \leq \frac{1}{2^n}\alpha(A)$ since $T^{-1}(D) \subset D.$ Therefore, we just need to control the values of $\alpha(T^{-n}(1))$ and $\alpha(T^{-n}(2))$ for each $n\in \N.$ But since $T^{-n}(2) = T^{-(n+1)}(1)$ it is actually enough to show that for appropriate choices of $\delta$, $\alpha_1=\alpha(1)$ and $\alpha_2 = \alpha(2)$ we can make the dynamical system $(\N, 2^{\N}, T, \alpha)$ power bounded in $L^1(\alpha).$ 
 Simple computations show that 
$T^{-1}(1)= 2$, $T^{-2}(1) = \{1,4\}$, $T^{-3}(1) = \{T^{-1}(1), T^{-1}(4)\} = \{2, T^{-1}(4)\}$.  By induction one can show that 
$T^{-n}(1) \subset \{1,2\}\cup_{j=0}^{n-2} T^{-j}(\{4\}).$ Indeed assuming this is true for $n$ we have 
$T^{-(n+1)}(\{1,2\}) \subset T^{-1}(\{1,2\})\cup_{j=1}^{n-1}T^{-j}(\{4\}) \subset \{1,2\}\cup_{j=0}^{n-1} T^{-j}(\{4\}).$ 
Setting $\alpha_1 = \alpha_2 = \delta$ we see that
$\alpha(T^{-n}(1)) \leq \alpha_1\sum_{j=0}^{n-2}\frac{1}{2^j} \leq 2 \alpha_1$.
 Combining this with our previous estimates one can conclude that for each $n\in \N$ and for each $A\subset \N$ we have 
$\alpha(T^{-n}(A)) \leq 2 \alpha(A).$

   \end{proof}
     
We can also derive the following result which characterizes the boundedness of the trajectories of the Collatz map 
   \begin{theorem}\label{PB2}
   Let  $(\N, 2^{\N}, T, \mu)$ be the Collatz dynamical system with the counting measure $\mu$. The following are equivalent.
   \begin{enumerate}
\item  There exists a finite measure $\alpha$ equivalent to $\mu$ for which the dynamical system $(\N, 2^{\N}, T, \alpha)$ is power bounded in $L^1(\alpha).$  
\item   The set $D_2$ is empty. 
\item  The trajectory of each point $n\in \N$ is bounded.
\item  For every bounded $f$ on $\N$, the averages $\frac{1}{N}\sum_{n=1}^N f(T^nx)$ converge for every $x\in \N.$
\end{enumerate}
\end{theorem}
\begin{proof}
The equivalence of the second and third statements in the theorem follow from Theorem \ref{HDC}. The first statement implies the second because all points in the dissipative part eventually enter one of the cycles $C_i.$
Therefore, their trajectories are bounded and the set $D_2$ must be empty. Conversely if $D_2$ is empty then $\N$ can be partitioned into the disjoint sets $F_j= \cup_{i=1}^{\infty} T^{-i}(C_j).$ Using the same method as in the proof of Theorem \ref{CJ} we can define on each set $F_j$ a measure $\nu_j$ such that for any $A_j\subset F_j$ we have $\nu_j(T^{-n}(A_j) ) \leq 2 \nu_j(A_j)$ for each positive integer $n\in \N.$ We define now $\alpha = \sum_{j=1}^{\infty} \frac{1}{2^j} \nu_j$. One can check that $\nu(T^{-n}(A) ) \leq 2 \nu(A) $ for each subset $A$ of $\N$. Thus the dynamical system $(\N, 2^{\N}, T, \alpha)$ is power bounded in $L^1(\alpha).$ This shows that the first and second statements are equivalent. 
It remains to show that the last statement is equivalent to any of the other three statements. It is enough to show that it is equivalent to the second statement 
  By \cite{GH} the last statement implies the second . For the reverse implication one can observe that Theorem \ref{Th3.3} shows that there exists a finite invariant measure with support $C$. This with the second statement implies the fourth by \cite{GH}.  
This completes the proof.

\end{proof}
   We have the following corollary.

\begin{cor}
 The following are equivalent for the Collatz dynamical system $(\N, 2^{\N}, \mu, T)$ 
\begin{enumerate} 
\item For each $n\in \N$ there exists $k$ such that $T^kn \in \{1,2\}$
\item For every bounded $f$ and $x\in\N$ we have $\frac{1}{N}\sum_{n=1}^N f(T^nx) \rightarrow \frac{1}{2}\left(f(1) + f(2)\right).$
\end{enumerate}
\end{cor} 
\begin{proof}
 The first statement implies  that $D_2$ is empty and that the only cycle is $\{1,2\}.$
By the fourth statement of the previous theorem for each boundef $f$ on $\N$ the averages 
$$\frac{1}{N}\sum_{n=1}^N f(T^nx) $$ converge for each $x\in\N.$ It remains to identify the limit to obtain the second statement. But this a consequence of the fact that starting at $x$ there exists a natural number m(x) such $T^{m(x)}x =1$. For $ k>m(x)$ the terms $T^kx$ alternate between $2$ and $1$. This gives the limit of the averages as being equal to $\frac{f(1) +f(2)}{2}.$\\
    For the converse, the second statement shows that $\{1,2\}$ is the only cycle. If not, taking $x$ in another cycle with period $p$, $C'= \{a, Ta, ..., T^{p-1}a\},$  and applying the second statement with $f= \mathbf{1}_{C'}$ we would get  $\frac{\sum_{k=1}^p f(T^kx) }{p}=1\neq 0 = \frac{f(1) +f(2)}{2}.$ Since the averages for $f$ bounded on $\N$ converge for each $x\in\N$ this implies that the set $D_2$ is empty by Theorem \ref{PB2} and proves the first statement.   
\end{proof}

\end{document}